\documentclass[12pt, twoside]{article}
\usepackage{bbm}
\usepackage{openbib}
\usepackage{amsmath,amsthm,amssymb,mathrsfs,amsfonts,mathbbold}
\usepackage{times,lineno}
\usepackage{enumerate}
\usepackage{tikz}
\usepackage{cite}
\usepackage[square,numbers,compress]{natbib}
\usepackage{pgflibraryarrows}
\usepackage{pgflibrarysnakes}							
\usepackage{hyperref}
\allowdisplaybreaks

\makeatletter
\def\author#1{\gdef\autrun{\def\and{\unskip, }#1}\gdef\@author{#1}}

\makeatother

\usepackage{makeidx}
\newtheorem{theorem}{Theorem}[section]

\newtheorem{lemma}[theorem]{Lemma}

\newtheorem{proposition}[theorem]{Proposition}
\theoremstyle{definition}

\newtheorem{remark}[theorem]{Remark}

\numberwithin{equation}{section}

\allowdisplaybreaks

 \makeatletter
\definecolor{ForestGreen}{rgb}{0.15,0.416,0.18}
\definecolor{EgyptBlue}{rgb}{0.063,0.2,0.65}
\hypersetup{
	colorlinks=true,
	linkcolor=EgyptBlue,
   citecolor=EgyptBlue,
   urlcolor=ForestGreen
}

\begin{document}


\title{Classification of stationary solutions of the $(4 + 1)$-dimensional radial Yang-Mills equation}

\author{Kui Li$^{a}$~~~and~~~Zhitao Zhang$^{b,c,}$\thanks{Corresponding author. This work is supported by  National Key R\&D Program of China (No. 2022YFA1005601), and by National
Natural Science Foundation of China (No. 12031015)}\\
{\small $^a$  School of Mathematics and Statistics, Zhengzhou University, Zhengzhou 450001,}\\
{\small P. R. China. E-mail: likui@zzu.edu.cn}\\
{\small $^b$  HLM, Academy of Mathematics and Systems Science, Chinese Academy}\\
{\small  of Sciences, Beijing 100190,  P. R. China. E-mail: zzt@math.ac.cn}\\
 {\small $^c$  School of Mathematical Sciences, University of Chinese Academy of Sciences,}\\
 {\small Beijing 100049,  P. R. China.}\\
 }
\date{}
\maketitle

\begin{abstract}
We classify the solutions of the planar weighted Allen-Cahn equation arising from the critical equivariant $SO(4)$ Yang-Mills problem. We first give a complete classification for radially symmetric solutions, and then without symmetry assumptions, we show bounded classical solutions are radially symmetric and also classify them.
\end{abstract}

\section{Introduction}

\quad\ \; In this paper, we consider the following weighted semilinear elliptic equation
\begin{equation}\label{wac}
 -\Delta u = \frac{2}{|x|^2}u\big(1-u^2\big)~~~~in~~~\mathbb{R}^2\backslash\{0\}.
\end{equation}\par

Equation \eqref{wac} arises from the studying of the critical equivariant $SO(4)$ Yang-Mills problem. Suppose that the gauge potential $A_\alpha$ is $SO(4)$-valued for $\alpha=0,1,\cdots,4$, and the curvature $F_{\alpha\beta}$ is given by $F_{\alpha\beta}=\partial_\alpha A_\beta-\partial_\beta A_\alpha+[A_\alpha, A_\beta]$
with $[\cdot,\cdot]$ be the Lie bracket on $SO(4)$ and $\alpha,\beta=0,1,\cdots,4$.
Then in $(4 + 1)$-dimensional Minkowski spacetime the critical Yang-Mills equations take the form
\begin{equation}\label{ym}
\begin{aligned}
&\partial_\alpha F^{\alpha\beta}+\big[A_\alpha, F^{\alpha\beta}\big]=0.
\end{aligned}
\end{equation}
Here the standard convention for raising indices is used.  Assume the radially symmetric ansatz:
\begin{equation}\label{ssa}
A_\alpha^{ij}(t,x)= (\delta_\alpha^ix_j-\delta_\alpha^jx_i)\frac{1-u(t,r)}{r^2},
\end{equation}
where $r=|x|$, $\alpha=0,1,\cdots,4$ and $i,j=1,\cdots,4$. Then we obtain the $(4 + 1)$-dimensional radial Yang-Mills equation (see \cite{D})
\begin{equation}\label{wwac}
\Box u = \frac{2}{r^2}u\big(1-u^2\big)~~~\mbox{in}~~\mathbb{R}\times \mathbb{R}^2
\end{equation}
with $\Box=\partial_{tt}-\Delta=\partial_{tt}-\partial_{rr}-\frac{1}{r}\partial_r$. When considering stationary solutions, equation \eqref{wwac} becomes equation \eqref{wac}.\par

In order to study the critical Yang-Mills problem, a great amount of effort has been devoted to equations of type \eqref{wac} and  type \eqref{wwac}. Note that equation \eqref{wac} admits one-parameter family of solutions, namely
\begin{equation}\label{ops}
 \frac{\lambda^2-|x|^2}{\lambda^2+|x|^2},~~~~~~\lambda\in \mathbb{R}.
\end{equation}
Making use of solutions of type \eqref{ops}, Schlatter et al. \cite{SST} showed the global existence of the critical equivariant Yang-Mills heat flow, Grotowski et al. \cite{GS} obtained a blow-up result for the critical equivariant Yang-Mills heat flow, C\^{o}te et al. \cite{CKM} completely characterized the behavior of solutions of the radial $4$D Yang-Mills equation as time goes to $\pm\infty$, Krieger et al. \cite{KST} proved that the critical Yang-Mills equations \eqref{ym} admit a family of solutions which blow up in finite time with the self-similar rate by a power of $|\log t|$, and in \cite{RR} Rapha\"{e}l et al. exhibited stable finite time blow up regimes for critical Yang-Mills equations \eqref{ym}, derived sharp asymptotics on the dynamics at blow up time and proved quantization of the energy focused at the singularity. In \cite{D1}, Donninger also used solutions of type \eqref{ops} of $(5+1)$ dimensional radial Yang-Mills equation to study the stable self-similar blowup for supercritical Yang-Mills problems. For other related work see \cite{KM}, \cite{KT1}, \cite{KT2}, \cite{OT1,OT2,OT3,OT4,OT5} and \cite{T}.\par


Equation \eqref{wac} is a natural extension of the Allen-Cahn equation
\begin{equation}\label{ac}
 -\Delta u =u\big(1-u^2\big).
\end{equation}
Equation \eqref{ac} originates from the gradient theory of phase transitions (see \cite{AC1} and \cite{CH}), and has deep connections with minimal surfaces (see \cite{CC} and \cite{M}). One of the most interesting and important questions concerning the classification of solutions of the equation \eqref{ac} is {\bf De Giorgi's conjecture (see \cite{CW} and \cite{DG}):}
\emph{Let $u$ be a bounded solution of the Allen-Cahn equation \eqref{ac} in $\mathbb{R}^N$, which is monotone in one direction, say $\frac{\partial u}{\partial x_N}>0$. Then, at least when $N\leq 8$, $u$ is one-dimensional.}\par

Ghoussoub and Gui \cite{GG} showed that De Giorgi's conjecture is true for $N=2$, Ambrosio and Cabr\'{e} \cite{AC} proved this conjecture for $N=3$. Under the additional limit condition
\begin{equation}\label{alc}
 \lim\limits_{x_N\rightarrow\pm\infty}u(x_1,\cdots,x_N)=\pm1,
\end{equation}
it is showed by Savin in \cite{S} that this conjecture holds true for $4\leq N\leq 8$. For $N\geq 9$, counterexamples were constructed by del Pino et al. \cite{DKW}. In \cite{W}, Wang gave a new and variational proof of the above Savin's theorem on this conjecture. Without the condition \eqref{alc}, this conjecture is still open for dimensions $4\leq N\leq 8$.\par

The aims of this paper are to classify the solutions of equation \eqref{wac} with symmetric assumption or without symmetric assumption. Our first result is to give a complete classification for radially symmetric solutions.

\begin{theorem}\label{thm1}
Let $u\in C^2\big(\mathbb{R}^2\backslash\{0\}\big)$ be a radially symmetric solution of equation \eqref{wac}. \\
$(i)$ If $u$ is continuous at the origin, then
\begin{equation*}
\mbox{either }~~u\equiv0~~\mbox{or}~~u(x)=\pm\frac{a^2-|x|^2}{a^2+|x|^2}~~\mbox{in}~~\mathbb{R}^2
\end{equation*}
with $a\in \mathbb{R}$.\\
$(ii)$ If $u$ is discontinuous at the origin, then the function  $v(t):=u(e^t)$ is a period function in $\mathbb{R}$ . Moreover, up to a translation, $v(t)$ can be characterized by its maximum value $M$ and minimum value $-M$ with $M\in(0,1)$ and the
period $T$ being given by the integral
\begin{equation}\label{e1.8}
T=2\int_{-M}^M\frac{\mathrm{d}\theta}{\sqrt{(2-M^2-\theta^2)(M^2-\theta^2)}}.
\end{equation}
\end{theorem}

\begin{remark}\label{rmk1}
(1) By Theorem \ref{thm1}, radially symmetric solutions of equation \eqref{wac} are all bounded and have the same bound $M=1$.\\
(2) We define the energy functional of equation \eqref{wac}  by
\begin{equation}\label{e1.10}
E(u)=\int_{\mathbb{R}^2}\big[|\nabla u|^2+\frac{(u^2-1)^2}{|x|^2}\big]\mathrm{d}x,
\end{equation}
then radially symmetric solutions of equation \eqref{wac} with finite energy must be
\begin{equation}\label{e1.10}
 u(x)=\pm\frac{a^2-|x|^2}{a^2+|x|^2}~~\mbox{in}~~\mathbb{R}^2,
 \end{equation}
where $a$ is a real constant.
\end{remark}

Next, we classify the solutions of equation \eqref{wac} without the radially symmetric assumption. Since the weight function of equation \eqref{wac} is singular at the origin, classical solutions of equation \eqref{wac} will be considered in the class
$$C^2(\mathbb{R}^2\backslash\{0\})\cap C(\mathbb{R}^2).$$
For classical solutions, we have the following classification results.
\begin{theorem}\label{thm2}
Let $u$ be a bounded classical solution of equation \eqref{wac}, then $u$ is radially symmetric and
\begin{equation*}
u\equiv0~~\mbox{or}~~u(x)=\pm\frac{a^2-|x|^2}{a^2+|x|^2}~~\mbox{in}~~\mathbb{R}^2
\end{equation*}
for some $a\in \mathbb{R}$.
\end{theorem}\par

\begin{remark}\label{rmk2}
By Theorem \ref{thm2} and Remark \ref{rmk1}, we deduce that nontrivial bounded classical solutions of equation \eqref{wac} have finite energy.
\end{remark}

The ideas of proofs of Theorem \ref{thm1} are as follows. We first apply the Emden type transform to equation \eqref{wac} and get an equivalent equation which is one-dimensional Allen-Cahn equation (see equation \eqref{equationforv}). For this equivalent equation, we only need to classify solutions that are not finite energy (see Remark \ref{rmkfiniteenergy}). We first use the first integral method to obtain a first order differential identity (see Lemma \ref{lemidentity}). By this identity and some integral techniques, we prove that solutions of the equivalent equation are bounded (see Lemma \ref{lemboundforv}). Using the boundness results for solutions, we show that the constant $c$ in the identity belongs to the interval $[-1,0]$ (see Lemma \ref{lemboundforc}). Also making full use of classical ODE theory, some integral techniques, the boundness results for solutions and the constant $c$, we completely classify the solutions of the equivalent equation for $c=-1$ (see Lemma \ref{classification1}) and for $c\in(-1,0)$ (see Lemma \ref{classification2}) respectively. By these classification results, we show Theorem \ref{thm1} holds. \par

In order to prove Theorem \ref{thm2}, we first use some integral techniques to analyze the values of solutions at the origin (see Lemma \ref{lemzero}), and use the Emden transform to obtain an equivalent Allen-Cahn equation with period condition (see Lemma \ref{lemv}). Apply the elliptic theory to this equivalent equation with period condition, we show that bounded classical solutions have bound $M=1$ (see Lemma \ref{lembdd}).  If the value of solutions to equation \eqref{wac} at the origin is equal to $0$, then by some integral techniques and ODE techniques we prove that solutions of equivalent equation are identically zero and thus obtain Theorem \ref{thm2} (see Lemma \ref{lemequalzero}). If the value of solutions of equation \eqref{wac} at the origin is equal to $1$ or $-1$, then by the bound $M=1$ and the methods of moving planes we show that solutions of equivalent equation are either symmetric with respect to some plane $t=T$ or strictly monotonic with respect to the variable $t$. In the first case, we also use some integral techniques and ODE techniques to prove that solutions of equivalent equation  are are identically $1$ or identically $-1$ (see Lemma \ref{lemone}). In the second case, we use a Liouville type theorem of Berestycki, Caffarelli and Nirenberg to prove that solutions of equivalent equation  are independent of $\theta$ (see Lemma \ref{lemsym}). Making use of Lemma \ref{lemone} and Lemma \ref{lemsym}, we prove Theorem \ref{thm2}.
\par

The rest of this paper is organized as follows. In section 2, we give some preliminaries; in section 3, we classify radially symmetric solutions and prove Theorem \ref{thm1}; and finally in section 4, we study  bounded classical solutions and prove Theorem \ref{thm2}.

\section{Preliminaries}
\numberwithin{equation}{section}
 \setcounter{equation}{0}

In this section, we give some preliminaries that will be useful in the proofs of our main theorems.

\begin{lemma}\label{lemzero}
Suppose that $u\in C^2\big(\mathbb{R}^2\backslash\{0\}\big)$ is a solution of equation \eqref{wac} with $\lim\limits_{x\rightarrow 0}u(x)=a\in \mathbb{R}$. Then $a=\pm 1$ or $a=0$.
\end{lemma}
\begin{proof} Suppose that the conclusion doesn't hold. Then there exists $l\in \mathbb{R}\backslash\{0\}$ such that
\begin{equation}\label{e2.1}
l=a\big(1-a^2\big).
\end{equation}
Without loss of generality, we may assume that $l>0$.\par

For any $r>0$, we define
$$\overline{u}(r)=\frac{1}{|S_1(0)|}\int_{S_1(0)}u\big(r\theta)\mathrm{d}\theta,$$
where $S_1(0)=\{x\in\mathbb{R}^2|~ |x|=1\}$.\par
Then direct calculations imply that
\begin{equation}\label{e2.2}
-\overline{u}''-\frac{1}{r}\overline{u}'=\frac{2}{r^2}\overline{u(1-u^2\big)}~~~\mbox{in}~~~(0,\infty).
\end{equation}\par

By \eqref{e2.1} and \eqref{e2.2}, we conclude that there exists a small positive constant $r_0$ such that for any $r\in(0,r_0)$,
\begin{equation}\label{e2.3}
-\overline{u}''-\frac{1}{r}\overline{u}'\geq\frac{l}{r^2}~~\mbox{and}~~-\big(r\overline{u}'\big)'\geq \frac{l}{r}.
\end{equation}
For any $r\in(0,r_0)$, integrating from $r$ to $r_0$ yields
\begin{equation}\label{e2.4}
r\overline{u}'(r)-r_0\overline{u}'(r_0)\geq l\int_r^{r_0}\frac{1}{\rho}d\rho=l(\ln r_0-\ln r)
\end{equation}
and
\begin{equation}\label{e2.5}
r\overline{u}'(r)\geq r_0\overline{u}'(r_0)+l\ln r_0-l\ln r.
\end{equation}
By \eqref{e2.5}, we deduce the existence of a positive constant $r_1\in (0, r_0)$ such that for any $r\in(0,r_1)$,
\begin{equation}\label{e2.6}
r\overline{u}'(r)\geq 1~~\mbox{and}~~\overline{u}'(r)\geq \frac{1}{r}.
\end{equation}
Hence for any $0<\varepsilon<r_1$, integrating from $\varepsilon$ to $r_1$ yields
$$\overline{u}(r)-\overline{u}(\varepsilon)\geq \int_{\varepsilon}^{r_1}\frac{1}{\rho}d\rho=\ln r_1-\ln \varepsilon.$$
Thus
$$\overline{u}(r_1)-a=\lim \limits_{\varepsilon\rightarrow0^+}[\overline{u}(r_1)-\overline{u}(\varepsilon)]\geq \lim \limits_{\varepsilon\rightarrow0^+}(\ln r_1-\ln \varepsilon)=\infty,$$
which is a contradiction, hence $l=0$ and we obtain this lemma.
\end{proof}

By the well known results about Kelvin transform, we have
\begin{lemma}\label{lemkel}
Suppose that $u\in C^2\big(\mathbb{R}^2\backslash\{0\}\big)$ is a solution of equation \eqref{wac}. If we define
\begin{equation*}
\tilde{u}(x)=u\big(\frac{x}{|x|^2}\big),~~x\in\mathbb{R}^2\backslash\{0\},
\end{equation*}
then $\tilde{u}(x)$ is also a solution of equation \eqref{wac}.
\end{lemma}

For any $\theta, t\in \mathbb{R}$, we define $r=e^t$ and
\begin{equation}\label{e2.7}
v(t,\theta)=u(r\cos \theta, r\sin \theta).
\end{equation}
Then by direct calculations for the Emden transform \eqref{e2.7}, we have the following result.

\begin{lemma}\label{lemv}
If $u\in C^2\big(\mathbb{R}^2\backslash\{0\}\big)$ is a solution of equation \eqref{wac}, then
\begin{equation}\label{e2.8}
-\partial_{tt}v-\partial_{\theta\theta}v=2v\big(1-v^2\big)
\end{equation}
with $v(t,\theta)=v(t,\theta+2\pi)$ for any $\theta, t\in \mathbb{R}$.
\end{lemma}

\begin{remark}\label{equivalent}
In following sections, we will make use of equation \eqref{e2.8} to classify solutions of equation \eqref{wac}.
\end{remark}

In the last part of this section, we give a bound for the bounded solutions of equation \eqref{wac}.
\begin{lemma}\label{lembdd}
If $u\in C^2\big(\mathbb{R}^2\backslash\{0\}\big)$ is a bounded solution of equation \eqref{wac}, then
$$|u(x)|\leq 1,~~~x\in \mathbb{R}^2.$$
\end{lemma}

\begin{proof}
Let
$$M=\sup\limits_{x\in\mathbb{R}^2}u(x)$$
and $v$ be define as in \eqref{e2.7}. Then
$$\sup\limits_{(t,\theta)\in\mathbb{R}^2}v(t,\theta)=M$$
 and $v$ is a solution of equation \eqref{e2.8}. Since the nonlinear term in equation \eqref{wac} is odd, we only need to show that $M\leq 1$.\par

If $M$ is obtained by $v$ at some point in $\mathbb{R}^2$, then applying Maximum principle to equation \eqref{e2.8}, we deduce that $2M\big(1-M^2\big)\geq0$ and hence $M\leq 1$.\par

If $v(t,\theta)< M$ for all $(t,\theta)\in\mathbb{R}^2$, then there exist $\big\{\theta_n\big\}_{n=1}^\infty\subseteq[0,2\pi]$ with $\lim\limits_{n\rightarrow\infty}\theta_n=\theta_0\in[0,2\pi]$, and $\big\{t_n\big\}_{n=1}^\infty\subseteq\mathbb{R}$ with $\lim\limits_{n\rightarrow\infty}t_n=\infty$ or $\lim\limits_{n\rightarrow\infty}t_n=-\infty$ such that
$$\lim\limits_{n\rightarrow\infty}v(t_n,\theta_n)=M.$$
Let $v_n(t,\theta)=v(t+t_n,\theta)$ for $(t,\theta)\in\mathbb{R}^2$. Then $\big\{v_n\big\}_{n=1}^\infty$ are uniformly bounded and all solve equation \eqref{e2.8}. By elliptic theory, there exists $v_0\in C^2\big(\mathbb{R}^2\big)$ such that
\begin{equation}\label{e2.9}
v_n\rightarrow v_0~\mbox{locally~in~}C^2\big(\mathbb{R}^2\big).
\end{equation}
Hence $v_0$ also solves equation \eqref{e2.8} and obtains its maximum $M$ at $\big(0,\theta_0\big)$, which implies that $M\leq 1$ by the previous arguments.
\end{proof}

\section{Classification of radially symmetric solutions}
\numberwithin{equation}{section}
 \setcounter{equation}{0}

\quad\ \; Let $u$ be a radially symmetric solution of equation \eqref{wac}, and $v(t)=u(e^t)$ for $t\in\mathbb{R}$. Then by Lemma \ref{lemv}, we deduce that
\begin{equation}\label{equationforv}
-v_{tt}=2v\big(1-v^2\big)~~\mbox{in}~~\mathbb{R},
\end{equation}
which is one-dimensional Allen-Cahn equation. Multiplying both sides of equation \eqref{equationforv} by $v_t$ and applying the first integral method, we have the following identity.
\begin{lemma}\label{lemidentity}
There exists  a constant $c\in \mathbb{R}$ such that
\begin{equation}\label{e3.2}
(v_t)^2=(v^2-1)^2+c~~\mbox{in}~~\mathbb{R}.
\end{equation}
\end{lemma}\par

\begin{remark}\label{rmkfiniteenergy}
For equation \eqref{equationforv}, we can define its energy functional
\begin{equation*}
E_0(v)=\frac{1}{2}\int_\mathbb{R}\big[(v_t)^2+(v^2-1)^2\big]\mathrm{d}t.
\end{equation*}
Then the solution $v$ of equation \eqref{equationforv} is finite energy if only if $c$ in \eqref{e3.2} is equal to $0$. Hence for finite energy solution $v$, we have
\begin{equation*}
(v_t)^2=(v^2-1)^2~~\mbox{in}~~\mathbb{R},
\end{equation*}
which implies the following well-known classification:
\begin{equation}\label{e3.11}
  v=\pm\frac{a^2-e^{2t}}{a^2+e^{2t}}~~\mbox{in}~\mathbb{R},
  \end{equation}
where $a$ is a real constant.
\end{remark}

To classify  radially symmetric solutions of equation \eqref{wac} without the finite energy assumption and the boundness assupmtion, we will make full use the identity \eqref{e3.2}. 

\begin{lemma}\label{lemboundforv}
Let $v$ be a solution of equation \eqref{equationforv}. Then $|v(t)|\leq 1$ for all $t\in\mathbb{R}$.
\end{lemma}
\begin{proof}
Suppose that $v$ is unbounded in $\mathbb{R}$. Without loss of generality, we assume that there exists a sequence $\{t_n\}_{n=1}^\infty$ such that
\begin{equation}\label{e3.3}
\lim\limits_{n\rightarrow\infty}t_n=-\infty~~\mbox{and}~~\lim\limits_{n\rightarrow\infty}v(t_n)=-\infty
\end{equation}\par

If $\lim\limits_{t\rightarrow-\infty}v(t)=-\infty$, then  there exists $t_0<0$ such that $(v^2-1)^2+c> \frac{1}{4}v^4>0$ for all $t\leq t_0$. By Lemma \ref{lemidentity}, we have
\begin{equation}\label{e3.4}
v_t>\frac{v^2}{2}~\mbox{and}~(-v^{-1})_t>\frac{1}{2},~~\forall~t\leq t_0,
\end{equation}
which implies that
\begin{equation}\label{e3.5}
v^{-1}(t)-v^{-1}(t_0)>\frac{1}{2}(t_0-t),~~\forall~t<t_0.
\end{equation}
Letting $t\rightarrow -\infty$ in both sides of \eqref{e3.5}, by the fact that $\lim\limits_{t\rightarrow-\infty}v(t)=-\infty$ we conclude
\begin{equation}\label{e3.6}
-v^{-1}(t_0)=\lim\limits_{t\rightarrow-\infty}v^{-1}(t)-v^{-1}(t_0)\geq \lim\limits_{t\rightarrow-\infty}\frac{1}{2}(t_0-t)=\infty,
\end{equation}
which is a contradiction.  \par

Hence there exists another sequence $\{s_n\}_{n=1}^\infty$ such that
\begin{equation}\label{e3.7}
\lim\limits_{n\rightarrow\infty}s_n=-\infty,~~v'(s_n)=0,~~~\forall~n\in\mathbb{N}.
\end{equation}
and
\begin{equation}\label{e3.8}
\lim\limits_{n\rightarrow\infty}v(s_n)=-\infty.
\end{equation}
By Lemma \ref{lemidentity}, the sequence $\{v(s_n)\}_{n=1}^\infty$ is bounded, which contradicts with the limit \eqref{e3.8}. Hence $v$ and $u$ are bounded in $\mathbb{R}$. By Lemma \ref{lembdd}, we obtain this lemma.
\end{proof}

Using the boundness of $v$ given in Lemma \ref{lemboundforv}, we can give a bound for $c$ in \eqref{e3.2}.

\begin{lemma}\label{lemboundforc}
Let $c$ be the constant in \eqref{e3.2}. Then $c\in[-1, 0]$.
\end{lemma}
\begin{proof}
Firstly, we show that $c\leq0$. Otherwise, we suppose that $c>0$ and $a=\sqrt{c}$. Then by \eqref{e3.2}, $(v_t)^2\geq a^2$ in $\mathbb{R}$. Hence
\begin{equation}\label{e3.9}
v_t\geq a~\mbox{or}~v_t\leq -a~~\mbox{in}~~\mathbb{R}.
\end{equation}
We may assume that $v_t\geq a$ in $\mathbb{R}$ since the proof for the case of $v_t\leq -a$ is similar. Then $\lim\limits_{t\rightarrow-\infty}v(t)=-\infty$. According to the proof of Lemma \ref{lemboundforv}, we obtain a contradiction. Hence $c\leq0$.\par

Secondly, we prove that $c\geq -1$. Otherwise, we assume that $c< -1$. Then by Lemma \ref{lemidentity}, we have
\begin{equation}\label{e3.10}
(v_t)^2=(v^2-b-1)(v^2+b-1)~~\mbox{in}~~\mathbb{R}
\end{equation}
with $b=\sqrt{-c}>1$. By Lemma \ref{lemboundforv},  $v^2-b-1<0$ in $\mathbb{R}$. Note that $v^2+b-1>0$ in $\mathbb{R}$, hence $(v^2-b-1)(v^2+b-1)<0$ in $\mathbb{R}$, which contradicts with equation \eqref{e3.10}. Hence  $c\geq -1$.
\end{proof}

Now we classify solutions of equation \eqref{equationforv} according the values of $c$ given in Lemma \ref{lemboundforc}.
\begin{lemma}\label{classification1}
Let $c$ be the constant in \eqref{e3.2}. If $c=-1$, then $v=0$ in $\mathbb{R}$.
\end{lemma}

\begin{proof} If $c=-1$, then by Lemma \ref{lemidentity} we have
\begin{equation}\label{e3.12}
(v_t)^2=v^2(v^2-2)~~\mbox{in}~\mathbb{R}.
\end{equation}
By Lemma \ref{lemboundforv}, $v^2-2<0$ in $\mathbb{R}$. Hence by \eqref{e3.12}, we deduce that $v=0$ in $\mathbb{R}$.\par


\end{proof}

Next we classify solutions of equation \eqref{equationforv} for $c \in (-1,0)$.
\begin{lemma}\label{classification2} If $c\in(-1,0)$, then  $v(t)$ is a period function in $\mathbb{R}$, and up to a translation, $v(t)$ can be characterized by its maximum value $M$ and minimum value $-M$ with $M\in(0,1)$.
\end{lemma}

\begin{proof} We divide the proof into five steps.\par
$Step~1$. Define $M=\sqrt{1-\sqrt{-c}}\in(0,1)$. Then by Lemma \ref{lemidentity}, we have
\begin{equation}\label{e3.16}
(v_t)^2=[v^2-(2-M^2)](v^2-M^2)~~\mbox{in}~\mathbb{R}.
\end{equation}
By Lemma \ref{lemboundforv}, $v^2-(2-M^2)<0$ in $\mathbb{R}$. Hence by \eqref{e3.16}, we deduce that $v\in [-M, M]$ in $\mathbb{R}$ and
\begin{equation}\label{e3.17}
(v_t)^2\geq 2(1-M^2)(M^2-v^2)~~\mbox{in}~\mathbb{R}.
\end{equation}\par

If $|v|<M$ in $\mathbb{R}$, then by \eqref{e3.16} and \eqref{e3.17} we have
\begin{equation}\label{e3.18}
v_t\geq \sqrt{2(1-M^2)}\sqrt{M^2-v^2}~\mbox{or}~v_t\leq- \sqrt{2(1-M^2)}\sqrt{M^2-v^2}~\mbox{in}~\mathbb{R}.
\end{equation}
We only consider the case of $v_t\geq \sqrt{2(1-M^2)}\sqrt{M^2-v^2}$ in $\mathbb{R}$. By integrating, we have
\begin{equation}\label{e3.19}
\arcsin\frac{v(t)}{M}\geq\sqrt{2(1-M^2)}t+c_0~\mbox{in}~\mathbb{R}
\end{equation}
for some $c_0\in\mathbb{R}$, which is impossible. Hence there exists $t_0\in\mathbb{R}$ such that $|v(t_0)|=M$. Without loss of generality, we assume that that $t_0=0$ and $v(0)=M$.
\par

By \eqref{e3.16}, $v_t(0)=0$. Applying the classical ODE theory to equation \eqref{equationforv}, we deduce that
\begin{equation}\label{e3.20}
v(-t)=v(t),~~\forall t\in\mathbb{R}.
\end{equation}\par

$Step~2$. We show that there exists $t_1>0$ such that $v(t_1)=-M$. Otherwise, by Lemma \ref{lemidentity} and the fact that $v_{tt}(0)=-2M(1-M^2)<0$ we have
\begin{equation}\label{e3.21}
-M<v(t)< M,~~v_t< 0,~~\forall t>0.
\end{equation}
and
\begin{equation}\label{e3.22}
v_t\leq -\sqrt{2(1-M^2)}\sqrt{M^2-v^2},~~\forall t>0.
\end{equation}
By \eqref{e3.21}, there exists $\ell\in[-M,M)$ such that $\lim\limits_{t\rightarrow\infty}v(t)=\ell$. Then by \eqref{e3.22},
\begin{equation}\label{e3.23}
\limsup\limits_{t\rightarrow\infty}v_t\leq -\sqrt{2(1-M^2)}\sqrt{M^2-\ell^2},
\end{equation}
which and the fact that $\lim\limits_{t\rightarrow\infty}v(t)=\ell\in[-M,M)$ imply that $\ell=-M$. 
By equation \eqref{e3.16}, we deduce that
\begin{equation}\label{e3.24}
\lim\limits_{t\rightarrow\infty}v_t=0
\end{equation}
But by equation \eqref{equationforv}, we have
\begin{equation}\label{e3.25}
\lim\limits_{t\rightarrow\infty}v_{tt}=2M\big(1-M^2\big)>0,
\end{equation}
which contradicts equation \eqref{e3.24}. Thus there exists $t_1>0$ such that $v(t_1)=-M$.\par

Note that by \eqref{e3.16}, $v_t(t_1)=0$. By the classical ODE theory, we conclude from equation \eqref{equationforv} that
\begin{equation}\label{e3.26}
v(2t_1-t)=v(t),~~\forall t\in\mathbb{R}.
\end{equation}\par

$Step~3$. By \eqref{e3.20} and \eqref{e3.26}, we conclude that $v(t)$ is a period function in $\mathbb{R}$ with maximum value $M$, minimum value $-M$ and $M\in(0,1)$. \par

$Step~4$. Suppose that $\bar{v}(t)$ is another period function in $\mathbb{R}$ with maximum value $M$, minimum value $-M$ and $M\in(0,1)$. Then there exists $t_2\in\mathbb{R}$ with $\bar{v}(t_2)=M$ and $\bar{v}_t(t_2)=0$.  Also applying the classical ODE theory to equation \eqref{equationforv}, we deduce that
\begin{equation}\label{e3.27}
\bar{v}(t)=v(t+t_2),~~\forall t\in\mathbb{R},
\end{equation}
which implies that $\bar{v}$ is a translation of $v$.\par

$Step~5$.  For any $M\in(0,1)$, consider the problem \eqref{equationforv} with initial values $ v(0)=M$ and $v_t(0)=0$. Then by \eqref{e3.2} with $c=-(1-M^2)^2$, \eqref{e3.16} and $Step~2$, we deduce that the solution $v(t)$ of equation \eqref{equationforv} exists globally in  $\mathbb{R}$ which is a period function with maximum value $M$ and minimum value $-M$.\par

By $Step~3$, $Step~4$ and $Step~5$, we prove this lemma.
\end{proof}

\begin{remark}\label{rmkperiod} According to the proof of Lemma \ref{classification2}, we know that the period $T$ of $v$ is given by
\begin{equation}\label{e3.28}
\begin{aligned}
T&=-2\int_{0}^{t_1}\frac{\mathrm{d}v(t)}{\sqrt{[v^2(t)-(2-M^2)](v^2(t)-M^2)}}\\
&=2\int_{-M}^M\frac{\mathrm{d}\theta}{\sqrt{(2-M^2-\theta^2)(M^2-\theta^2)}}.
\end{aligned}
\end{equation}
\end{remark}

Now, we turn to the proof of Theorem \ref{thm1}.\par
{\bf Proof~of~Theorem \ref{thm1}:}\par
 (1) If $u$ is continuous at the origin, then by Lemma \ref{lemzero} we know that $u(0)=0$ or $u(0)=\pm1$. Hence $\lim\limits_{t\rightarrow-\infty}v(t)=0$ or $\lim\limits_{t\rightarrow-\infty}v(t)=\pm1$. By classical ODE theory, we know that $\lim\limits_{t\rightarrow-\infty}v_t(t)=0$. Therefore, by Lemma \ref{lemidentity} we have
\begin{equation}\label{e3.28}
c=\lim\limits_{t\rightarrow-\infty}[(v_t)^2-(v^2-1)^2]=-1~\mbox{or}~c=\lim\limits_{t\rightarrow-\infty}[(v_t)^2-(v^2-1)^2]=0.
\end{equation}
Then by Lemma \ref{classification1} and Remark \ref{rmkfiniteenergy}, we deduce that there exists $a\in \mathbb{R}$ such that

\begin{equation}\label{e3.29}
v\equiv0~~\mbox{or}~~v(x)=\pm\frac{a^2-e^{2r}}{a^2+e^{2r}}~~\mbox{in}~~\mathbb{R},
\end{equation}
which implies that
\begin{equation}\label{e3.30}
\mbox{either }~~u\equiv0~~\mbox{or}~~u(x)=\pm\frac{a^2-|x|^2}{a^2+|x|^2}~~\mbox{in}~~\mathbb{R}^2
\end{equation}
with $a\in \mathbb{R}$.\par

(2) If $u$ is not continuous at the origin, then by Lemma \ref{lemboundforc} we know that $c\in(-1,0)$.  Otherwise, $c=-1$ or $c=0$. Then we conclude from the previous arguments that $u$ must be  continuous at the origin, which is a contradiction. Hence by Lemma \ref{classification2} and Remark \ref{rmkperiod}, we also obtain this theorem.\qed

\section{Classification of bounded classical solutions}
\numberwithin{equation}{section}
 \setcounter{equation}{0}
\quad\ \; In order to classify the bounded classical solutions, we first study equation \eqref{e2.8}, and renumber it as follows:
\begin{equation}\label{e4.1}
-\partial_{tt}v-\partial_{\theta\theta}v=2v\big(1-v^2\big)
\end{equation}
with $v(t,\theta)=v(t,\theta+2\pi)$ for any $\theta, t\in \mathbb{R}$.
\begin{proposition}\label{propmoving}
Suppose that $v \in C^2\big(\mathbb{R}^2\big)$ is a solution to the equation \eqref{e4.1}. If $\lim\limits_{t\rightarrow\infty}v(t,\theta)=1$ uniformly for $\theta\in \mathbb{R}$, $v(t,\theta)=v(t, \theta+2\pi)$ and $|v(t,\theta)|\leq 1$ for any $(t,\theta)\in \mathbb{R}^2$, then
\begin{description}
  \item (i) $v_t(t,\theta)>0$ for any $(t,\theta)\in \mathbb{R}^2$ or
  \item (ii) $\exists \Lambda \in \mathbb{R}$ such that
  \begin{equation}\label{e4.2}
  v(t, \theta)=v(2\Lambda-t, \theta), ~~\forall(t,\theta)\in \mathbb{R}^2.
  \end{equation}
\end{description}
\end{proposition}

We will use the methods of moving planes to prove Proposition \ref{propmoving}. Before we give the proof, we need some notations.\par

For any $\lambda \in \mathbb{R}$, we define (see \cite{CL})
\begin{equation}\label{e4.3}
\begin{aligned}
&\Sigma_\lambda=\big\{(t,\theta): t>\lambda,~\theta\in \mathbb{R}\big\},~~&T_\lambda=\big\{(t,\theta): t=\lambda,~\theta\in \mathbb{R}\big\},\\
&v^\lambda(t,\theta)=v(2\lambda-t,\theta),~~~~&w^\lambda(t,\theta)=v(t,\theta)-v^\lambda(t,\theta).
\end{aligned}
\end{equation}
Then
\begin{equation}\label{e4.4}
w^\lambda(t,\theta)=0,~~\forall (t,\theta)\in T_\lambda.
\end{equation}
By direct calculations, we have
\begin{equation}\label{e4.5}
-\Delta v^\lambda=2v^\lambda\big(1-(v^\lambda)^2\big)
\end{equation}
and
\begin{equation}\label{e4.6}
-\Delta w^\lambda=-\Delta(v-v^\lambda)=2\big(1-3(\xi^\lambda)^2\big)w^\lambda,
\end{equation}
where $\xi^\lambda=\xi^\lambda(t,\theta)$ is between $v(t,\theta)$ and $v^\lambda(t,\theta)$.
\par

Let
\begin{equation}\label{e4.7}
\Sigma_\lambda^-=\Sigma_\lambda\cap \big\{w^\lambda<0\big\}.
\end{equation}
Since $\lim\limits_{t\rightarrow\infty}v(t,\theta)=1$ uniformly for $\theta\in \mathbb{R}$ and $|v(t,\theta)|\leq 1$ for any $(t,\theta)\in \mathbb{R}^2$, we deduce that
\begin{equation}\label{e4.8}
\varliminf \limits_{t\rightarrow \infty} w^\lambda(t,\theta)\geq 0,~~\forall \theta\in \mathbb{R}
\end{equation}
and there exists $\lambda_0\gg 1$ such that
\begin{equation}\label{e4.9}
v(t, \theta)>\frac{1}{\sqrt{3}},~~~~\forall \theta\in \mathbb{R}, ~\forall t>\lambda_0.
\end{equation}
Hence
\begin{equation}\label{e4.10}
\frac{1}{\sqrt{3}}<v<\xi^\lambda<v^\lambda,~~~3\xi^2-1>0,~~~~\forall (t, \theta)\in \Sigma_\lambda^-, ~\forall t>\lambda_0.
\end{equation}

\begin{lemma}\label{lemstart}
$v\geq v^\lambda$ in $\Sigma_{\lambda}$ for any $\lambda>\lambda_0$.
\end{lemma}

\begin{proof}
By \eqref{e4.4}, \eqref{e4.6} and \eqref{e4.8}, we conclude that $w^\lambda$ solves the equation
\begin{equation}\label{e4.11}
\left\{
\begin{aligned}
    -\Delta w^\lambda+2\big(3(\xi^\lambda)^2-1\big)w^\lambda&=0~~&\mbox{in}&~\Sigma_\lambda^-,\\
     w^\lambda(t, \theta+2\pi)&=w^\lambda(t, \theta),&\forall&(t,\theta)\in \Sigma_\lambda^-,\\
     w^\lambda(t,\theta)&= 0,~~&\forall& (t,\theta)\in \partial \Sigma_\lambda^-,\\
     \varliminf \limits_{(t,\theta)\in\Sigma_\lambda^-,~t\rightarrow \infty} w^\lambda(t,\theta)&\geq 0,\\
\end{aligned}
\right.
\end{equation}\par

Since $\lambda>\lambda_0$, by \eqref{e4.10} and Maximum principle we conclude that
\begin{equation*}
w^\lambda\geq0~~~\mbox{in}~ \Sigma_\lambda^-.
\end{equation*}
Note that $w^\lambda<0$ in $\Sigma_\lambda^-$, hence
\begin{equation*}
\Sigma_\lambda^-=\emptyset,
\end{equation*}
which implies that
\begin{equation}\label{e4.12}
w^\lambda\geq0~~~\mbox{and}~~v\geq v^\lambda~~~\mbox{in}~ \Sigma_\lambda
\end{equation}
for any $\lambda>\lambda_0$.
\end{proof}

Similarly to \cite{CL}, we define
\begin{equation}\label{e4.13}
\Lambda=\inf\big\{\lambda\in \mathbb{R}|v\geq v^\lambda~\mbox{in~}\Sigma_\lambda\big\}.
\end{equation}
By Lemma \ref{lemstart}, we have
\begin{equation*}
\Lambda\in \big[-\infty, \lambda_0\big].
\end{equation*}

\begin{lemma}\label{lemcom}
If $\Lambda \in \mathbb{R}$, then $w^\Lambda\equiv0$ in $\Sigma_\Lambda$.
\end{lemma}
\begin{proof}
Since $v$ is continuous, $w^\Lambda$ solves the equation
\begin{equation}\label{e4.14}
\left\{
\begin{aligned}
    -\Delta w^\Lambda+2\big(3\xi^2-1\big)w^\Lambda&=0~~&\mbox{in}&~\Sigma_\Lambda,\\
     w^\Lambda&\geq0,~~&\mbox{in}&~\Sigma_\Lambda,\\
     w^\Lambda&=0,~~&\mbox{on}&~T_\Lambda.
\end{aligned}
\right.
\end{equation}
By Maximum principle, there are two cases occur:\par
(i) $w^\Lambda\equiv0$ in $\Sigma_\Lambda$ or \par
(ii) $w^\Lambda>0$ in $\Sigma_\Lambda$ and $\frac{\partial w^\Lambda}{\partial t}>0$ on $T_\Lambda$. \\
Hence in order to prove this lemma, we need to exclude case (ii). \par

In the following, we assume that case (ii) occurs and want to obtain a contrdiction. 
By the definition of $\Lambda$, for any $n\in \mathbb{N}^+$ we have
$$\inf \limits_{\Sigma_n} w_n<0,$$
where $\Sigma_n=\Sigma_{\Lambda-\frac{1}{n}}$ and $w_n=w^{\Lambda-\frac{1}{n}}$. Since $w_n=0$ on $T_n:=T_{\Lambda-\frac{1}{n}}$, $\varliminf \limits_{t\rightarrow \infty} w_n(t,\theta)\geq 0$ for any $\theta\in \mathbb{R}$, and $w_n(t,\theta)=w_n(t, \theta+2\pi)$ for any $(t,\theta)\in \mathbb{R}^2$, there exists $(t_n, \theta_n)\in \Sigma_n$ with $\theta_n\in [0,2\pi]$ such that
\begin{equation}\label{e4.15}
w_n(t_n, \theta_n)=\inf \limits_{\Sigma_n} w_n<0~~\mbox{and}~~\nabla w_n(t_n, \theta_n)=(0,0).
\end{equation}

Note that $w^\Lambda>0$ in $\Sigma_\Lambda$ and $w^\Lambda(t,\theta)=w^\Lambda(t, \theta+2\pi)$ for any $(t,\theta)\in \mathbb{R}^2$, then $\exists\delta_0>0$ such that
\begin{equation}\label{e4.16}
w^\Lambda>\delta_0~~\mbox{in}~~\Sigma_{\Lambda+1}\backslash\Sigma_{\lambda_0},
\end{equation}
where $\lambda_0$ is given in \eqref{e4.9}. Hence for $n\gg 1$,  we have
\begin{equation}\label{e4.17}
w_n>\delta_0~~\mbox{in}~~\Sigma_{\Lambda+1}\backslash\Sigma_{\lambda_0}.
\end{equation}\par

By \eqref{e4.6}, \eqref{e4.8},  \eqref{e4.10} and \eqref{e4.17}, we conclude that for any $n\gg1$, $w_n$ solves the equation
\begin{equation}\label{e4.18}
\left\{
\begin{aligned}
    -\Delta w_n+2\big(3\xi^2-1\big)w_n&=0~~&\mbox{in}&~\Sigma_n^-,\\
     w_n(t, \theta+2\pi)&=w_n(t, \theta),&\forall&(t,\theta)\in \Sigma_n^-,\\
     w_n(t,\theta)&\geq 0,~~&\forall& (t,\theta)\in \partial \Sigma_n^-,\\
     \varliminf \limits_{(t,\theta)\in\Sigma_n^-,~t\rightarrow \infty} w_n(t,\theta)&\geq 0,\\
\end{aligned}
\right.
\end{equation}
with $\Sigma_n^-=\Sigma_{\lambda_0}\cap\{w_n<0\}$ and $3\xi^2-1\geq 3v^2-1>0$ in $\Sigma_n^-$. By  Maximum principle and the fact that $w_n<0$ in $\Sigma_n^-$, we conclude that
\begin{equation}\label{e4.19}
\Sigma_n^-=\emptyset ~~\mbox{and}~~w_n\geq0~~~\mbox{in}~ \Sigma_{\lambda_0}.
\end{equation}
\par

By \eqref{e4.17} and \eqref{e4.19}, we have
\begin{equation}\label{e4.20}
w_n\geq0~~\mbox{in}~~\Sigma_{\Lambda+1}
\end{equation}
for any $n\gg 1$. Thus we conclude from \eqref{e4.15} and \eqref{e4.20} that $t_n\in \big[\Lambda, \Lambda+1\big]$ for $n\gg 1$ and $\big\{(t_n, \theta_n)\big\}_{n=1}^\infty$ is bounded. Up to a subsequence, we may assume that
\begin{equation}\label{e4.21}
(t_n, \theta_n)\rightarrow (t_0, \theta_0)~~~\mbox{as}~n\rightarrow \infty
\end{equation}
with $t_0\geq \Lambda$ and $\theta_0\in [0,2\pi]$.
\par

By the continuity of $w$ and \eqref{e4.21}, we have
\begin{equation*}
0\leq w^\Lambda(t_0, \theta_0)=\lim\limits_{n\rightarrow\infty}w_n(t_n, \theta_n)\leq 0, ~~ w^\Lambda(t_0, \theta_0)=0.
\end{equation*}
Hence
\begin{equation*}
t_0=\Lambda,~~~\nabla w^\Lambda(\Lambda, \theta_0)=\lim\limits_{n\rightarrow\infty}\nabla w_n(t_n, \theta_n)=0,
\end{equation*}
which contradicts with the fact that $\frac{\partial w^\Lambda}{\partial t}>0$ on $T_\Lambda$. Thus we exclude the case (ii).
\end{proof}

Making use of Lemma \ref{lemstart} and Lemma \ref{lemcom}, we can prove Proposition \ref{propmoving}.\par
\emph{\bf Proof~of~Proposition \ref{propmoving}:}\par
 Let $\Lambda$ be defined in \eqref{e4.13}. We first assume that $\Lambda \in \mathbb{R}$. Then by Lemma \ref{lemstart}, we have
\begin{equation*}
w^\Lambda(t, \theta)=0~~\mbox{and}~~v(t, \theta)=v(2\Lambda-t, \theta),~~\forall(t,\theta)\in \mathbb{R}^2,
\end{equation*}
which imply (ii).\par

Now, we assume that $\Lambda=-\infty$. By the definition of $\Lambda$, for any $\lambda\in \mathbb{R}$ we conclude that $w^\lambda\geq 0$ and $w^\lambda$ solves
\begin{equation*}
\left\{
\begin{aligned}
    -\Delta w^\lambda+2\big(3\xi^2-1\big)w^\lambda&=0~~&\mbox{in}&~\Sigma_\lambda,\\
     w^\lambda&=0,~~&\mbox{on}&~T_\lambda.
\end{aligned}
\right.
\end{equation*}
Hence by Maximum principle, we deduce that either $w^\lambda\equiv0$ in $\Sigma_\lambda$ or $w^\lambda>0$ in $\Sigma_\lambda$ and $\frac{\partial w^\lambda}{\partial t}>0$ on $T_\lambda$. If $w^{\lambda}\equiv0$, then we obtain (ii) with $\Lambda=\lambda$. Otherwise, $\frac{\partial w^\lambda}{\partial t}=2v_t>0$ in $T_\lambda$ for any $\lambda\in \mathbb{R}$, and we obtain (i).\qed

With the help of Proposition \ref{propmoving}, we can prove $v$ that satisfies conditions of Proposition  \ref{propmoving} is independent of $\theta$.
\begin{lemma}\label{lemone}
Suppose the conditions of Proposition \ref{propmoving} are satisfied. If $\exists \Lambda \in \mathbb{R}$ such that
$v(t, \theta)=v(2\Lambda-t, \theta)$ for any $(t,\theta)\in \mathbb{R}^2$, then
\begin{equation*}
v\equiv 1~~\mbox{in}~\mathbb{R}^2.
\end{equation*}
\end{lemma}
\begin{proof} Since $v$ is symmetric with respect to the plane $t=\Lambda$, we have
\begin{equation}\label{e4.22}
v_t(\Lambda, \theta)=0, ~~\forall\theta\in \mathbb{R}.
\end{equation}\par

Multiplying both sides of equation \eqref{e4.1} by $v_t$ and integrating over $[0,2\pi]$ with respect to $\theta$, we have
\begin{equation}\label{e4.23}
-\int_0^{2\pi}v_{tt}v_t\mathrm{d}\theta-\int_0^{2\pi}v_{\theta\theta}v_t\mathrm{d}\theta
=2\int_0^{2\pi}vv_t\mathrm{d}\theta-2\int_0^{2\pi}v^3v_t\mathrm{d}\theta.
\end{equation}
Hence
\begin{equation*}
\frac{\mathrm{d}}{\mathrm{d}t}\big[-\frac{1}{2}\int_0^{2\pi}v_t^2\mathrm{d}\theta+\frac{1}{2}\int_0^{2\pi}v_{\theta}^2\mathrm{d}\theta\big]
=\frac{\mathrm{d}}{\mathrm{d}t}\big[\int_0^{2\pi}v^2\mathrm{d}\theta-\frac{1}{2}\int_0^{2\pi}v^4\mathrm{d}\theta\big]
\end{equation*}
and there exists constant $c_1\in \mathbb{R}$ such that
\begin{equation}\label{e4.24}
-\frac{1}{2}\int_0^{2\pi}v_t^2\mathrm{d}\theta+\frac{1}{2}\int_0^{2\pi}v_{\theta}^2\mathrm{d}\theta
=\int_0^{2\pi}v^2\mathrm{d}\theta-\frac{1}{2}\int_0^{2\pi}v^4\mathrm{d}\theta+c_1~~~~\mbox{in}~\mathbb{R}.
\end{equation}
\par

Note that $\lim\limits_{t\rightarrow\infty}v(t,\theta)=1$ uniformly for $\theta\in \mathbb{R}$, by elliptic theory
\begin{equation}\label{e4.25}
\lim\limits_{t\rightarrow\infty}v_t(t,\theta)=\lim\limits_{t\rightarrow\infty}v_\theta(t,\theta)=0
\end{equation}
uniformly for $\theta\in \mathbb{R}$. Letting $t\rightarrow\infty$ in both sides of equation \eqref{e4.24}, by \eqref{e4.25} we have
\begin{equation}\label{e4.26}
0=\lim\limits_{t\rightarrow\infty}LHS=\lim\limits_{t\rightarrow\infty}RHS=\pi+c_1~~\mbox{and}~~c_1=-\pi.
\end{equation}
Hence by \eqref{e4.24} and \eqref{e4.26}, we have
\begin{equation}\label{e4.27}
\frac{1}{2\pi}\int_0^{2\pi}v_{\theta}^2\mathrm{d}\theta+\frac{1}{2\pi}\int_0^{2\pi}v^4\mathrm{d}\theta+1
=\frac{1}{2\pi}\int_0^{2\pi}v_t^2\mathrm{d}\theta+\frac{2}{2\pi}\int_0^{2\pi}v^2\mathrm{d}\theta
\end{equation}
for any  $t\in \mathbb{R}$.\par

Now, fix $t=\Lambda$. By \eqref{e4.22}, we have
\begin{equation}\label{e4.28}
\frac{1}{2\pi}\int_0^{2\pi}v_t^2\mathrm{d}\theta=0
\end{equation}
and by H\"{o}lder inequality, we have
\begin{equation}\label{e4.29}
\frac{1}{2\pi}\int_0^{2\pi}v^2\mathrm{d}\theta\leq\bigg(\frac{1}{2\pi}\int_0^{2\pi}v^4\mathrm{d}\theta\bigg)^{\frac{1}{2}}.
\end{equation}
Then by \eqref{e4.27}, \eqref{e4.28} and \eqref{e4.29}, we deduce that
\begin{equation}\label{e4.30}
\frac{1}{2\pi}\int_0^{2\pi}v_{\theta}^2\mathrm{d}\theta+\bigg[\big(\frac{1}{2\pi}\int_0^{2\pi}v^4\mathrm{d}\theta\big)^{\frac{1}{2}}-1
\bigg]^2\leq 0
\end{equation}
at $t=\Lambda$. Hence at $t=\Lambda$, by \eqref{e4.30} and the fact that $|v(t,\theta)|\leq 1$ for any $(t,\theta)\in \mathbb{R}^2$ we have
\begin{equation}\label{e4.31}
v_{\theta}\equiv 0,~~\frac{1}{2\pi}\int_0^{2\pi}v^4\mathrm{d}\theta=1~~\mbox{and}~~v\equiv\pm1.
\end{equation}\par

If $v\equiv-1$ at $t=\Lambda$, then by Maximum principle and the facts that $|v(t,\theta)|\leq 1$ for any $(t,\theta)\in \mathbb{R}^2$ and $v_0\equiv-1$ is also a solution to equation \eqref{e4.1}, we deduce that
\begin{equation*}
v\equiv -1~~\mbox{in}~~\mathbb{R}^2,
\end{equation*}
which contradicts with the fact that $\lim\limits_{t\rightarrow\infty}v(t,\theta)=1$ uniformly for $\theta\in \mathbb{R}$. Thus $v\equiv1$ at $t=\Lambda$, and by Maximum principle we have
\begin{equation*}
v\equiv 1~~\mbox{in}~~\mathbb{R}^2,
\end{equation*}
which implies this lemma.
\end{proof}

\begin{lemma}\label{lemsym}
Suppose the conditions of Proposition \ref{propmoving} are satisfied. If $v_t(t,\theta)>0$ for any $(t,\theta)\in \mathbb{R}^2$, then
\begin{equation*}
v(t,\theta)=v(t)~~\mbox{in}~\mathbb{R}^2.
\end{equation*}
\end{lemma}
\begin{proof}
Differentiating both sides of equation \eqref{e4.1} with respect to $t$, we conclude that $v_t$ is a positive solution to the equation
\begin{equation}\label{e4.32}
-\Delta z=2\big(1-3v^2\big)z
\end{equation}
in the $(t,\theta)-$ plane with $z(t,\theta)=z(t,\theta+2\pi)$ for any $\theta, t\in \mathbb{R}$. It is obvious that $v_\theta$ also solves the equation \eqref{e4.32}. Hence
\begin{equation*}
\nabla \bigg(\frac{v_\theta}{v_t}\bigg)=\frac{v_t\nabla v_\theta-v_\theta\nabla v_t}{v_t^2}
\end{equation*}
and
\begin{equation}\label{e4.33}
\begin{aligned}
-div\bigg[v_t^2\nabla \big(\frac{v_\theta}{v_t}\big)\bigg]&=-div\big(v_t\nabla v_\theta-v_\theta\nabla v_t\big)\\
&=-v_t\Delta v_\theta+v_\theta\Delta v_t\\
&=2\big(1-3v^2\big)v_tv_\theta-2\big(1-3v^2\big)v_tv_\theta=0~~~\mbox{in}~\mathbb{R}^2.
\end{aligned}
\end{equation}
By a Liouville type theorem in \cite{BCN} for $\mathbb{R}^2$, we know that there exists $c_2\in \mathbb{R}$ such that
\begin{equation}\label{e4.34}
\frac{v_\theta}{v_t}\equiv c_2 ~~~\mbox{in}~\mathbb{R}^2.
\end{equation}
Since $v(t,\cdot)$ is $2\pi$ period for any $t\in \mathbb{R}$, $v_\theta$ has zeros. Hence $c_2$ in \eqref{e4.34} is equal to $0$ and
\begin{equation*}
v_\theta\equiv 0 ~~~\mbox{in}~\mathbb{R}^2,
\end{equation*}
which implies $v(t,\theta)=v(t)$ in $\mathbb{R}^2$.
\end{proof}

In the end of this section, we study the bounded classical solutions which take value $0$ at the origin.

\begin{lemma}\label{lemequalzero}
Suppose $u$ is a bounded classical solution of equation \eqref{wac} with $u(0)=0$. Then
\begin{equation*}
u\equiv 0~~~\mbox{in}~\mathbb{R}^2.
\end{equation*}
\end{lemma}

\begin{proof}
Let $v$ be defined as in \eqref{e2.7}. Then similar to the proof of Lemma \ref{lemone}, we deduce that there exists constant $c_3\in \mathbb{R}$ such that
\begin{equation}\label{e4.35}
-\int_0^{2\pi}v_t^2\mathrm{d}\theta+\int_0^{2\pi}v_{\theta}^2\mathrm{d}\theta
=2\int_0^{2\pi}v^2\mathrm{d}\theta-\int_0^{2\pi}v^4\mathrm{d}\theta+c_3~~~~\mbox{in}~\mathbb{R}.
\end{equation}
\par

Note that $\lim\limits_{t\rightarrow-\infty}v(t,\theta)=0$ uniformly for $\theta\in \mathbb{R}$, by elliptic theory
\begin{equation}\label{e4.36}
\lim\limits_{t\rightarrow-\infty}v_t(t,\theta)=\lim\limits_{t\rightarrow-\infty}v_\theta(t,\theta)=0
\end{equation}
uniformly for $\theta\in \mathbb{R}$. Letting $t\rightarrow-\infty$ in both sides of equation \eqref{e4.35} and by \eqref{e4.36}, we have
\begin{equation}\label{e4.37}
0=\lim\limits_{t\rightarrow-\infty}LHS=\lim\limits_{t\rightarrow-\infty}RHS=c_3~~\mbox{and}~~c_3=0.
\end{equation}
Hence by \eqref{e4.35} and \eqref{e4.37}, we have
\begin{equation}\label{e4.38}
-\int_0^{2\pi}v_t^2\mathrm{d}\theta+\int_0^{2\pi}v_{\theta}^2\mathrm{d}\theta
=2\int_0^{2\pi}v^2\mathrm{d}\theta-\int_0^{2\pi}v^4\mathrm{d}\theta
\end{equation}
for any  $t\in \mathbb{R}$. Thus for any  $a, r\in \mathbb{R}$ with $a\gg 1$, by \eqref{e4.38} we have
\begin{equation}\label{e4.39}
-\int_{-a}^t\int_0^{2\pi}v_\rho^2\mathrm{d}\rho\mathrm{d}\theta+\int_{-a}^t\int_0^{2\pi}v_{\theta}^2\mathrm{d}\rho\mathrm{d}\theta
=2\int_{-a}^t\int_0^{2\pi}v^2\mathrm{d}\rho\mathrm{d}\theta-\int_{-a}^t\int_0^{2\pi}v^4\mathrm{d}\rho\mathrm{d}\theta.
\end{equation}\par

Multiplying both sides of equation \eqref{e4.1} by $v$ and integrating over $[-a, t]\times [0, 2\pi]$ with respect to $\rho$ and $\theta$, we have
\begin{equation}\label{e4.40}
-\int_{-a}^t\int_0^{2\pi}\big[v_{\rho\rho}v+v_{\theta\theta}v\big]\mathrm{d}\rho\mathrm{d}\theta
=2\int_{-a}^t\int_0^{2\pi}v^2\mathrm{d}\rho\mathrm{d}\theta-2\int_{-a}^t\int_0^{2\pi}v^4\mathrm{d}\rho\mathrm{d}\theta.
\end{equation}
By integrating by parts, we have
\begin{equation}\label{e4.41}
-\int_{-a}^t\int_0^{2\pi}v_{\theta\theta}v\mathrm{d}\rho\mathrm{d}\theta=\int_{-a}^t\int_0^{2\pi}v_\theta^2\mathrm{d}\rho\mathrm{d}\theta
\end{equation}
and
\begin{equation}\label{e4.42}
\begin{aligned}
-\int_{-a}^t\int_0^{2\pi}v_{\rho\rho}v\mathrm{d}\rho\mathrm{d}\theta
&=-\int_0^{2\pi}\big[vv_\rho\big|^t_{-a}-\int_{-a}^tv_\rho^2\mathrm{d}\rho\big]\mathrm{d}\theta\\
&=\int_0^{2\pi}\int_{-a}^tv_\rho^2\mathrm{d}\rho\mathrm{d}\theta-\int_0^{2\pi}vv_\rho\mathrm{d}\theta\bigg|^t_{-a}.
\end{aligned}
\end{equation}
By \eqref{e4.40}, \eqref{e4.41} and \eqref{e4.42}, we have
\begin{equation}\label{e4.43}
\int_{-a}^t\int_0^{2\pi}\big[v_\rho^2+v_\theta^2\big]\mathrm{d}\rho\mathrm{d}\theta
=2\int_{-a}^t\int_0^{2\pi}v^2\mathrm{d}\rho\mathrm{d}\theta
-2\int_{-a}^t\int_0^{2\pi}v^4\mathrm{d}\rho\mathrm{d}\theta+\frac{1}{2}\varphi'(\rho)\big|^t_{-a},
\end{equation}
where
\begin{equation}\label{e4.44}
\varphi(\rho)=\int_0^{2\pi}v^2\mathrm{d}\theta.
\end{equation}\par

By \eqref{e4.39} and \eqref{e4.43}, we conclude that
\begin{equation*}
2\int_{-a}^t\int_0^{2\pi}v_\rho^2\mathrm{d}\rho\mathrm{d}\theta
=-\int_{-a}^t\int_0^{2\pi}v^4\mathrm{d}\rho\mathrm{d}\theta+\frac{1}{2}\varphi'(\rho)\big|^t_{-a}
\end{equation*}
and
\begin{equation}\label{e4.45}
\varphi'(t)=4\int_{-a}^t\int_0^{2\pi}v_\rho^2\mathrm{d}\rho\mathrm{d}\theta
+2\int_{-a}^t\int_0^{2\pi}v^4\mathrm{d}\rho\mathrm{d}\theta+\varphi'(-a)
\end{equation}
for any $a\gg 1$ and $t\in \mathbb{R}$. By \eqref{e4.36} and the fact that $v$ is bounded, we have
\begin{equation*}
\lim\limits_{\rho\rightarrow-\infty}\varphi'(\rho)=\lim\limits_{\rho\rightarrow-\infty}2\int_0^{2\pi}vv_\rho\mathrm{d}\theta=0,
\end{equation*}
Hence letting $a\rightarrow\infty$ in \eqref{e4.45}, we have
\begin{equation}\label{e4.46}
\varphi'(t)=4\int_{-\infty}^t\int_0^{2\pi}v_\rho^2\mathrm{d}\rho\mathrm{d}\theta
+2\int_{-\infty}^t\int_0^{2\pi}v^4\mathrm{d}\rho\mathrm{d}\theta\geq 0
\end{equation}
for any $t\in \mathbb{R}$, and that $\varphi$ is increasing in $\mathbb{R}$. In the following, we divide the rest of the proofs into two case.\par

\emph{Case 1}: $\int_{-\infty}^\infty\int_0^{2\pi}v^4\mathrm{d}\rho\mathrm{d}\theta<\infty$.  Then there exists a sequence $\big\{t_n\big\}^\infty_{n=1}\subseteq\mathbb{R}$ with $\lim\limits_{n\rightarrow\infty}t_n=\infty$ such that
\begin{equation*}
\lim\limits_{n\rightarrow\infty}\int_0^{2\pi}v^4(t_n, \theta)\mathrm{d}\theta=0.
\end{equation*}
By H\"{o}lder inequality, we have
\begin{equation*}
\lim\limits_{n\rightarrow\infty}\varphi(t_n)=0,
\end{equation*}
which combined with the fact that $\varphi$ is increasing in $\mathbb{R}$ yields that
\begin{equation}\label{e4.47}
\varphi\equiv0~~~\mbox{in}~\mathbb{R}.
\end{equation}
Hence
\begin{equation*}
v\equiv 0~~\mbox{and}~~u\equiv 0~~~\mbox{in}~\mathbb{R}^2.
\end{equation*}\par

\emph{Case 2}: $\int_{-\infty}^\infty\int_0^{2\pi}v^4\mathrm{d}\rho\mathrm{d}\theta=\infty$. Since $v$ is bounded in $\mathbb{R}^2$, $\int_0^{2\pi}v^4\mathrm{d}\theta$ is bounded in $\mathbb{R}$. Thus for $t\gg1$,
\begin{equation*}
\int_{-\infty}^t\int_0^{2\pi}v^4\mathrm{d}\rho\mathrm{d}\theta
\geq\int_0^{2\pi}v^4\mathrm{d}\theta.
\end{equation*}
By \eqref{e4.46} and H\"{o}lder inequality, we have
\begin{equation}\label{e4.48}
\varphi'(t)\geq 2\int_0^{2\pi}v^4\mathrm{d}\theta\geq\frac{2}{2\pi}\bigg(\int_0^{2\pi}v^2\mathrm{d}\theta\bigg)^2=\frac{\varphi^2(t)}{\pi}
\end{equation}
for $t\gg1$.\par

If $\varphi(t_0)>0$, then the fact that $\varphi$ is increasing in $\mathbb{R}$ yields that $\varphi(t)>0$ for any $t>t_0$. Hence by \eqref{e4.48}, we deduce that there exists $c_4\in \mathbb{R}$ such that
\begin{equation*}
-\frac{1}{\varphi(t)}\geq\frac{t}{\pi}+c_4,~~~\frac{1}{\varphi(t)}+\frac{t}{\pi}+c_4\leq0
\end{equation*}
for any $t\gg 1$, which is impossible. Hence
\begin{equation*}
\varphi\equiv0~~~\mbox{in}~\mathbb{R},
\end{equation*}
which implies that
\begin{equation*}
v\equiv 0~~\mbox{and}~~u\equiv 0~~~\mbox{in}~\mathbb{R}^2.
\end{equation*}
\end{proof}

Now, we turn to the proof of Theorem \ref{thm2}.\par
\emph{\bf Proof~of~Theorem \ref{thm2}:}\par
 By Lemma \ref{lemzero}, we have
\begin{equation*}
u(0)=0~~ \mbox{or}~~u(0)=\pm 1.
\end{equation*}

\par If $u(0)=0$, then by Lemma \ref{lemequalzero} we have $u\equiv 0$. If $u(0)=1$, then we define
$$\tilde{u}(x)=u\bigg(\frac{x}{|x|^2}\bigg)~~\mbox{and}~~v(t, \theta)=\tilde{u}(r\cos \theta, r\sin \theta)$$
with $r=e^{-t}$. By Lemma \ref{lemkel}, Lemma \ref{lemv}, Proposition \ref{propmoving}, Lemma \ref{lemone} and Lemma \ref{lemsym}, we deduce that $v(t,\theta)=v(t)$, and thus $u$ is radially symmetric. Hence by Theorem \ref{thm1}, we have
$$u(x)=\frac{a^2-|x|^2}{a^2+|x|^2}~~\mbox{for~some~}a\in \mathbb{R}.$$

\par
If $u(0)=-1$, then $-u$ is also a solution of equation \eqref{wac} with $-u(0)=1$. By the above arguments, we have
$$u(x)=-\frac{a^2-|x|^2}{a^2+|x|^2}~~\mbox{for~some~}a\in \mathbb{R}.$$\qed
\section*{Acknowledgements}
\noindent

The authors thank the referees for their valuable suggestions.


\end{document}